\documentclass[11pt]{article}
\usepackage{amsfonts}
\usepackage[margin=2cm]{geometry}
\usepackage{graphicx,color}
\usepackage[hang,small,bf]{caption}
\usepackage{xspace}
\usepackage{enumerate}
\usepackage{amssymb,amsmath}
\usepackage{amsthm}

\usepackage[T1]{fontenc}
\usepackage[utf8]{inputenc}
\usepackage{authblk}

\newtheorem{thm}{Theorem}

\newtheorem{proposition}[thm]{Proposition}
\newtheorem{lemma}[thm]{Lemma}

\newtheorem{definition}[thm]{Definition}

\begin{document}

\title{On the Approximation Ratio of the \\ Random Chinese Postman Tour for Network Search }

\author[]{Thomas Lidbetter}
\affil[]{\emph{Department of Management Science and Information Systems, Rutgers Business School,} \\ {\em 1 Washington Park, Newark, NJ 07102, USA, tlidbetter@business.rutgers.edu}}

 \date{}
  \maketitle
  \thispagestyle{empty}
\maketitle

 \begin{abstract}
We consider a classic search problem first proposed by S. Gal in which a \emph{Searcher} randomizes between unit speed paths on a network, aiming to find a hidden point in minimal expected time in the worst case. This can be viewed as a zero-sum game between the Searcher and a time maximizing \emph{Hider}. It is a natural model for many search problems such as search and rescue operations; the search for an enemy, a bomb or weapons in a military context; or predator-prey search. A \emph{Chinese Postman Tour} (\emph{CPT}) is a minimal time tour of the network that searches all the arcs and a \emph{Random Chinese Postman Tour} (\emph{RCPT}) is an equiprobable choice of any given CPT and its reverse. The full class of networks for which a RCPT is optimal is known, but otherwise little is known about the solution of the game except in some special cases that have complicated optimal strategies that would be impractical to implement. The question of how well a RCPT or any other search strategy performs for general networks has never been analyzed. We show that a RCPT has an approximation ratio of $4/3$: that is, the maximum expected time it takes to find a point on the network is no greater than $4/3$ times that of the optimal search strategy. We then examine the performance of a RCPT in a related search game recently proposed by S. Alpern in which the Searcher must return to his starting point after finding the Hider. 

 \end{abstract}{\bf Keywords:} game theory; search games; Chinese Postman Tour; networks

\newpage
\setcounter{page}{1}

\section{Introduction}
\label{intro}

\cite{Gal79} proposed a search problem on a connected network consisting of nodes joined by arcs. A {\em Searcher}, beginning at a distinguished {\em root} node of the network moves around with unit speed with the aim of locating a target, or {\em Hider}, situated at an unknown point of the network. The problem for the Searcher is to find a randomized search plan, or {\em mixed strategy} that minimizes the expected time to locate the Hider in the worst case. This can be viewed as a zero-sum game between the time minimizing Searcher and the time maximizing Hider. The idea of this game was originally conceived by \cite{Isaacs}.

This is a natural model of many search problems in a military context, for example the problem of locating a terrorist, a bomb or a weapon cache hidden in a known environment. It is also relevant to search-and-rescue problems and to models of predator-prey interaction in the natural world. In security scenarios, randomization can be exploited to ensure that your actions are not predictable to your opponent. An example of this is the use of randomized security patrolling used in Los Angeles Airport (see \cite{Pita}). But ease of implementation is also an important factor to consider, especially in a military context. In this paper we strive to balance the desire for randomness and unpredictability with the desire for easily implemented search strategies.

In his original work, \cite{Gal79} gave a full solution of his game for tree networks (those with no cycles) and Eulerian networks (those containing an Eulerian tour, that is one that traverses each arc of the network exactly once). Gal found that for both these classes of networks it is optimal for the Searcher to identify a minimal time tour, called a {\em Chinese Postman Tour} (or {\em CPT}) of the network, and choose with equal probability the CPT or its reverse. This mixed strategy is known as a {\em Random Chinese Postman Tour} (or {\em RCPT}), and it is easy to see that the expected time taken for the Searcher using a RCPT to locate the Hider, wherever he is located on the network, will be no more than half the length $L$ of the CPT. This is because if a given CPT finds a point of a network at time $t$, the reverse of the CPT finds it at time no greater than $L-t$, so that the expected discovery time is at most $(1/2)t+(1/2)(L-t)=L/2$.

Since its original formulation the game described above has received considerable attention, for example in the work of \cite{Reijnierse-Potters}, \cite{Pavlovic} and \cite{Gal:2001}. In the latter work, Gal finally completely classified those networks for which it was optimal for the Searcher to employ a RCPT. The networks for which this is the case are known as {\em weakly Eulerian}, and are, roughly speaking, those consisting of a number of disjoint Eulerian cycles which, when each contracted to a single point leave a tree.

On networks that are not weakly Eulerian, little is known about the optimal search strategy in general. Even very simple networks can have very complicated optimal search strategies. For example, the solution of the game played on the ``3-arc network'', consisting of 2 nodes joined by 3 unit length arcs was discovered by \cite{Pavlovic93} not until several years after the model was first formulated. For this network, the optimal strategy for the Searcher finds the Hider in expected time $(4+\ln(2))/3 \approx 1.564$, and the optimal Searcher strategy is complicated, involving randomizing between a continuum of strategies. This would be impractical to implement, and so we may compare the performance of the optimal strategy of the Searcher to that of a RCPT, which is straightforward to implement. In order to do this we calculate the {\em approximation ratio} of a RCPT: that is, the ratio between the maximum expected time required by the RCPT to find a point on the network to that of an optimal strategy. Since the length of a CPT for the 3-arc network is $4$, a RCPT finds the Hider in expected time no more than $(1/2) \cdot 4 = 2$, so the RCPT has approximation ratio $2/1.564 \approx 1.28$ for this network. 

In this paper we analyze the effectiveness of the RCPT as a search strategy in the general case, giving a simple formula in Section \ref{sec:Gal} for calculating the approximation ratio of a RCPT and showing that it never exceeds $4/3$. It is well known (\cite{Edmonds-Johnson}) that the problem of finding a CPT is computable in polynomial time (cubic in the number of nodes of the network), and once it has been found, a RCPT tour is easily implemented in a practical search situation by means of a toss of a coin. This work shows that the RCPT is not only easily implementable but also reasonably efficient as a search strategy. Little attention has been paid in the literature to the problem of finding approximately optimal Searcher strategies in unsolved cases of this game, and for search games in general.

In Section \ref{sec:f&f} we go on to consider a natural variant on the original model in which the Searcher wishes to minimize not simply the time to find the Hider but the total time to find him and return to her starting point. The return speed may be different to the search speed. This model, known as {\em find-and-fetch search}, was introduced recently by \cite{Alperna}, and is a more appropriate model for search-and-rescue operations in which a casualty must be found and taken back to the hospital in least possible time. The model also pertains to foragers in the natural world who seek food to return to their lairs. This added detail complicates the solution to the model considerably, and in \cite{Alperna} the solution is given for only some classes of tree networks. Unlike in Gal's original model, the Searcher's optimal strategy in the find-and-fetch model is complicated and involves making a randomized decision at each node of the tree.  In this work we assess the performance of a RCPT in the find-and-fetch model for {\em all} tree networks, finding a lower bound for the value of the game by using a new technique we call ``pruning'' to produce a new tree on which the game is easier to analyze. We also study the game on Eulerian networks, again giving a simple formula for the approximation ratio of a RCPT.

This work lies in the general area of {\em search games}, on which there is an extensive literature.  For good summaries see \cite{AlpernGal}, \cite{Garnaev} and \cite{Search-th}. In addition to the find-and-fetch model of \cite{Alperna}, there have been a number of recent extensions to the original model of search games on a network proposed by \cite{Gal79}, for example work of \cite{DG} and \cite{Alpern-arb} on search games with an arbitrary starting point for the Searcher. Search games in which the Searcher must pay a {\em search cost} to inspect a node of the network are considered in \cite{BK:costs1} and \cite{BK:costs2}. The expanding search paradigm of \cite{AL:exp} models search problems in which the Searcher can move instantaneously back to any point he has already searched, and \cite{AL:small} consider a model in which the Searcher can move at either a slow (searching) speed or a fast (non-searching) speed. There has also been much recent interest in {\em patrolling games}, which are search games on a network where a Patroller wishes to intercept some terrorist attack, for example \cite{AMP}, \cite{BGA} and \cite{ZFZ}.

\section{Preliminaries}
\label{sec:prelim}
Let $Q$ be a network consisting of a finite connected set of arcs which intersect at nodes, with a distinguished root node $O$. We may think of a network as an edge weighted multigraph embedded in three-dimensional Euclidean space in such a way that the edges intersect only at nodes of the network. (Three dimensions are the minimum required in general for such an embedding to be possible.) The edge weights correspond to the linear Lebesgue measure $\mu$ of the arcs of the network, so that the measure (or {\em length}) of an arc $a$ is $\mu(a)$. We may consider any subset $A \subset Q$ of points of a network (which may not correspond to a subgraph of the original multigraph). If $A$ is measurable, we write $\mu(A)$ for its length, and we write the total length of $Q$ as $\mu(Q) = \mu$.

A \emph{search strategy} is a unit speed walk $S: \mathbb{R}^+ \rightarrow Q$ starting at $O$, so that $S(0)=O$. More precisely, for any $0 \le t_1 \le t_2$, we insist that $d(S(t_1),S(t_2)) \le t_2-t_1$, where $d$ is the distance function on pairs of points of $Q$ given by taking the length of the shortest path between them. Since we are thinking of $Q$ as being embedded in Euclidean space, arcs can be traversed in either direction, and the Searcher does not have to finish traversing an arc after she starts, but can turn around and backtrack at any point. A \emph{mixed search strategy}, usually denoted by a lower case letter $s$ is a probabilistic choice of strategies. For a given point $H$ on $Q$ and a given search strategy $S$, we denote the first time that $S$ reaches $H$ by $T(S,H)$, which we call the \emph{search time}. That is
\[
T(S,H)=\min\{t\ge 0: S(t) = H\}.
\]
Note that a hiding place $H$ can be anywhere on the network, including in the interior of an arc. The search time is known to be well defined (see \cite{AlpernGal}). We consider the problem of determining the mixed strategy that minimizes the expected time to find any point on $Q$ in the worst case. Writing the set of all mixed search strategies as $\mathcal{S}$, the problem is to determine 
\[
\inf_{s \in \mathcal{S}} \sup_{H \in Q} T(s,H),
\]
where $T(s,H)$ is the expected value of the search time of $H$ under $s$. This can be viewed as a zero-sum game $\Gamma = \Gamma(Q,O)$ between the Searcher, who chooses a search strategy and a malevolent Hider who picks a point on $Q$. The payoff, which the Searcher seeks to minimize and the Hider to maximize, is the search time. The game is known to have a value, $V=V(Q)$ (see \cite{AlpernGal}), and the players have optimal strategies, which in general are mixed (randomized). A mixed strategy for the Hider is a distribution over $Q$ and is usually denoted by a lower case letter $h$. For mixed strategies $s$ and $h$ of the Searcher and Hider, respectively, we write $T(s,h)$ for the expected value of the search time, which we called the {\em expected search time}.

\cite{Gal79} showed that if $Q$ is a tree or Eulerian, a RCPT is optimal for the Searcher, as explained in the Introduction, so that the value of the game is equal to $\bar{\mu}/2$, where $\bar{\mu}$ is the length of a CPT. In the case of trees, the value is $\bar{\mu}/2=\mu$ and in the case of Eulerian networks it is $\bar{\mu}/2=\mu/2$. The complete class of networks for which the RCPT is optimal was later shown by \cite{Gal:2001} to be the class of weakly Eulerian networks. We have already informally defined weakly Eulerian networks in the Introduction, but we give a formal definition here, which is easily seen to be equivalent. We use the notion of the {\em bridge-block tree} of a network, $Q$. This is obtained by considering an equivalence relation on the nodes of $Q$ which relates two nodes if there are two arc-disjoint paths connecting them. Note that every node is equivalent to itself, since the path containing only that node is arc-disjoint from itself. This partitions the nodes of $Q$ into equivalence classes, and for each equivalence class we may consider the sub-network containing all the nodes in that class and all arcs joining them. Each of these sub-networks is {\em 2-arc-connected}: that is, they cannot be disconnected without removing at least $2$ arcs. We call these subnetworks the {\em blocks} of $Q$. Sometimes $2$-arc-connected networks are referred to as \textit{bridgeless}, because they have no {\em bridges} (disconnecting arcs). The arcs of $Q$ that are incident to two nodes in different blocks of $Q$ are exactly the bridges of $Q$. The bridge-block tree of $Q$ is the tree that has a node for every block of $Q$ and an arc between two nodes exactly when there is a bridge between the associated block in $Q$.

\begin{definition}
	A network $Q$ is {\em weakly Eulerian} if all its blocks are Eulerian.
\end{definition}

We now state Gal's result formally.

\begin{thm} [Theorems 2 and 3 from \cite{Gal:2001}]
	\label{thm:weakly-Eul}
	The value of the game played on $Q$ is $\bar{\mu}/2$ and a RCPT is optimal if and only if $Q$ is weakly Eulerian.
\end{thm}

The only non-weakly Eulerian networks for which the solution of the game is known are those consisting of two nodes joined by an odd number of arcs of the same length (\cite{Pavlovic93}). The simplest example of a network of this type is the 3-arc network described in the Introduction. We saw that for this network, the expected search time of a RCPT is at most 1.28 times the value of the game. Since the value of the game is not known in general, we may use lower bounds on the value to estimate it. One lower bound is obtained from the observation that if the Hider hides uniformly in the network (so that the probability he is located in any subset $X$ of $Q$ is proportional to $\mu(X)$) then the expected search time is at least $\mu/2$. 

It is also easy to see that the length of a CPT is no more than $2\mu$, since by ``doubling'' every arc of the network we obtain an Eulerian network of length $2\mu$, and an Eulerian tour of this network maps onto a tour of the original network with the length $2\mu$. Combining these two estimates of the value $V$ yields Theorem 3.19 from \cite{AlpernGal}:
\begin{equation}
\mu/2 \le V \le \bar{\mu}/2 \le \mu. \label{eq:weak-est}
\end{equation}
We write the maximum expected time of a RCPT over all points on $Q$ as $T_{RCPT}$. We are interested in finding the approximation ratio of the RCPT, as described in the Introduction. In particular, we say the RCPT has approximation ratio $\alpha$ if $T_{RCPT} \le \alpha V$ for any network. By (\ref{eq:weak-est}), $T_{RCPT}$ satisfies
\[
T_{RCPT} \le \bar{\mu}/2 \le \mu \le 2V. 
\]
In other words, the RCPT has approximation ratio 2.

The question of whether the ratio of $2$ can be improved has never been considered, and in the next section we show that it can be improved to $4/3$.

\section{Gal's classic game}
\label{sec:Gal}

This section deals with the approximation ratio of the RCPT in Gal's classic search game. In the next section we consider the approximation ratio of the RCPT in the ``find-and-fetch'' model recently introduced by \cite{Alperna}.

We begin by stating our main theorem.

\begin{thm}
	\label{thm:main}
	The approximation ratio of the RCPT in Gal's classic game is $4/3$. That is, $T_{RCPT} \le (4/3)V$.
\end{thm}

To prove the theorem, we start by obtaining a better lower bound on the value of the game. We use an observation from \cite{Alpernb} which says that if two nodes are ``stuck together'', the value of the game cannot increase.

\begin{lemma}
	\label{lemma:iden}
	Let $u$ and $v$ be nodes of a network $Q$, and consider a new network $Q'$ obtained by identifying $u$ and $v$ (this can be thought of as adding an arc of length $0$ from $u$ to $v$). Then $V(Q') \le V(Q)$.
\end{lemma}

More precisely, \textit{identifying} nodes $u$ and $v$ means forming a new network $Q'$ where all arcs between $u$ and $v$ are removed and $u$ and $v$ are replaced by a single new node $w$; each arc in $Q$ with one endpoint in $\{u,v\}$ and other endpoint $x \notin \{u, v\}$ is replaced with an arc of the same length and endpoints $w$ and $x$.

Lemma \ref{lemma:iden} is easily seen to hold since the strategy set of the Hider in $Q'$ remains unchanged and any strategy of the Searcher in $Q$ can also be mapped onto a strategy in $Q'$ with no greater search time. We may also identify a larger set of nodes by successively identifying pairs of nodes in the set, and it clearly follows from Lemma~\ref{lemma:iden} that this will also not increase the value of the game.

Let $\mu_1$ be the total length of the bridge-block tree of $Q$ (that is, the sum of the lengths of all bridges of $Q$) and let $\mu_2$ be the sum of the lengths of the blocks. 

\begin{lemma}
	The value $V(Q)$ of the game on a network $Q$ satisfies
	\begin{equation}
	V(Q) \ge \mu_1 + \mu_2/2,
	\label{value-lb}
	\end{equation}
	with equality if $Q$ is weakly Eulerian.
\end{lemma}

\begin{proof}
We first define a new network $Q_{WE}$ with $V(Q_{WE}) \le V(Q)$. For every block $X$ of $Q$, identify all the nodes of $X$, to obtain an Eulerian network of the same measure. By Lemma \ref{lemma:iden}, the value $V(Q_{WE})$ of the game on $Q_{WE}$ is no greater than $V(Q)$. Also, $Q_{WE}$ is clearly weakly Eulerian, since all its blocks are Eulerian. The length of a CPT on $Q_{WE}$ is $2\mu_1+\mu_2$, so by Theorem \ref{thm:weakly-Eul},
\[
V(Q) \ge V(Q_{WE}) = (2\mu_1+\mu_2)/2 = \mu_1+\mu_2/2.
\]
If $Q$ is already weakly Eulerian, we clearly have equality in (\ref{value-lb}) since by Theorem \ref{thm:weakly-Eul}, the value $V(Q)$ is half the length of the CPT, which is $2\mu_1+\mu_2$. 
\end{proof}

We will also need a combinatorial result, taken from Proposition 5 of \cite{BJJ}, which we state below.

\begin{proposition}[Proposition 5 of \cite{BJJ}] \label{prop:4/3}
	Every connected bridgeless graph $G$ has a [Chinese] postman tour of length at most $\frac{4}{3} |E(G)|$.
\end{proposition}

Note that in \cite{BJJ}, a graph is defined in the usual way as a set $V(G)$ of vertices and a set $E(G)$ of edges, which consist of (unordered) pairs of vertices. This differs from our definition of networks, as given in Section~\ref{sec:prelim}, but Proposition~\ref{prop:4/3} easily extends to networks, as we show in Lemma~\ref{lemma:bridgeless}.

We first give a brief idea of how \cite{BJJ} prove Proposition~\ref{prop:4/3}. The result follows from a lemma proved in the same paper which says that for any bridgeless graph, it is possible to find a subset $C$ of edges that can be partitioned into cycles such that $|C| \ge \frac 2 3 |E(G)|$. By ``doubling'' the remaining edges in $E(G)-C$, we obtain an Eulerian graph with no more than $\frac 4 3 |E(G)|$ edges, and an Eulerian cycle on this new graph corresponds to a tour of length at most $\frac 4 3 |E(G)|$ in $G$.

\begin{lemma}
	\label{lemma:bridgeless}
	Suppose the network $Q$ is 2-arc-connected. Then the length $\bar{\mu}$ of a CPT of $Q$ satisfies
	\[
	\bar{\mu} \le (4/3)\mu.
	\]
	If $Q$ is any network, $\bar{\mu}$ satisfies
	\begin{equation}
	\bar{\mu} \le 2\mu_1 + (4/3)\mu_2.
	\label{CPT-ub}
	\end{equation}
\end{lemma}

\begin{proof}
The first statement follows from the analagous result of \cite{BJJ} for graphs, since additional nodes can be added to $Q$ to transform it into a network whose arcs are arbitrarily close in length.

The second statement is easily seen to be true since any CPT of $Q$ must traverse all the bridges twice. 
\end{proof}

Theorem \ref{thm:main} then follows, since the ratio of the expected search time $T_{RCPT}$ of a RCPT and the value $V$ of the game satisfies
\begin{align*}
\frac{T_{RCPT}}{V} &\le \frac{\bar{\mu}/2}{V}\\
& \le \frac{\mu_1 + (2/3)\mu_2}{\mu_1+ \mu_2/2} \mbox{ (by (\ref{value-lb}) and (\ref{CPT-ub})})\\
& \le 4/3.
\end{align*}

\section{The Find-and-fetch game}
\label{sec:f&f}

The find-and-fetch model, introduced recently by \cite{Alperna}, differs from Gal's classic model in one way: the Searcher must return to the root of the network at return speed $\rho$ after finding the Hider. More precisely, the payoff of the game for a Hider strategy $H$ and Searcher strategy $S$ is not $T(S,H)$ but $R(S,H):=T(S,H)+d(H)/\rho$, where $d(H)$ is the length of the shortest path from $H$ to $O$. As $\rho \rightarrow \infty$, the payoff approaches $T(S,H)$, as in Gal's game. In this section we use the same notation $V(Q)$ to denote the value of the find-and-fetch game played on a network $Q$ with root $O$.

We may obtain a crude estimate on the approximation ratio of the RCPT in the find-and-fetch game using the observations that $T(RCPT,H) \le \mu$ and $d(H) \le \mu$ for any Hider strategy $H$, so that $R(RCPT,H) \le \mu(1+1/\rho)$. Since the expected search time of any Searcher strategy against the uniform Hider strategy is at least $\mu/2$, we must have $V \ge \mu/2$. Hence, writing $R_{RCPT}$ for the maximum expected value over all points $H$ in $Q$ of $T_{RCPT}+d(H)/\rho$, we have 
\[
\frac{R_{RCPT}}{V} \le \frac{\mu(1+1/\rho)}{\mu/2} = 2(1+1/\rho),
\]
showing that a RCPT is a $2(1+1/\rho)$-approximation for return speed equal to $\rho$. Note that if $\rho$ is at least $1$, then the RCPT is a $4$-approximation.

We will analyze the efficiency of the RCPT for the find-and-fetch game on trees and Eulerian networks, finding better bounds in these cases.

\subsection{Tree Networks}

We will first state the theorem from \cite{Alperna} which gives the value of the game for a certain class of trees. In order to do this, we must define a distribution $e$ for the Hider, known as the {\em Equal Branch Density} (or {\em EBD}) distribution, on the leaf nodes, $\mathcal{L}$ of a tree $Q$. We say a point or arc $x$ in $Q$ is {\em above} another point or arc $y$ if the unique path from $x$ to $O$ contains $y$. If $u$ is a node of $Q$ with at least two incident arcs above it, we call $u$ a {\em branch node}, and we call each of those incident arcs {\em branch arcs}. For a branch node $u$, a {\em branch} at $u$ is a subtree rooted at $u$ consisting of a branch arc $a$ and all other points above $a$, and is denoted $Q_a$. We write $Q_u$ for the union of all the branches at $u$.

\begin{definition} The EBD distribution $e$ on $Q$ is the unique distribution on $\mathcal{L}$ such that for any branch node $u$ with branch $Q_a$ at $u$, the probability $e(Q_a)$ that the Hider is located on $Q_a$ is proportional to the length $\mu(Q_a)$ of $Q_a$. That is, $e(A)=\mu(Q_a)/\mu(Q_u)$.
	
	We also define $D=D(Q)$ to be the average distance of the leaf nodes from $O$, weighted with respect to the EBD distribution, so that
	\[
	D=\sum_{u \in \mathcal{L}} e(u) d(u).
	\]
\end{definition}

We summarize the relevant results from \cite{Alperna}.

\begin{thm}
	\label{thm:f&f}
	The value $V(Q)$ of the find-and-fetch game on a rooted network $Q$ satisfies
	\[
	V(Q) \ge \mu + D/\rho.
	\]
	The inequality above holds with equality if $\rho \ge \tilde{\rho}$, where $\tilde{\rho}$ is a constant which measures the irregularity of the tree.
\end{thm}

We refer the reader to \cite{Alperna} for a precise defininition of $\tilde{\rho}$, since it is not relevant to our analysis.

Since the expected time for the RCPT to reach any leaf node is exactly $\mu$, we have
\begin{equation}
R_{RCPT} = \mu + d_{\max}/\rho,
\label{eq:R-RCPT}
\end{equation}
where $d_{\max}$ is the distance of some farthest point in $Q$ from $O$.

In order to measure the efficiency of a RCPT we estimate how much larger $\mu + d_{\max}/\rho$ is than $\mu + D/\rho$, and we do this by obtaining a new tree $\tilde{Q}$ from $Q$ with $D(\tilde{Q}) \le D(Q)$ by means of a process we call {\em pruning}. If $u$ is a branch node of $Q$ and $a$ is a branch arc incident to $u$, let $Q[a]$ be the tree obtained by removing $Q_a$ and reattaching it at $O$. We call $Q[a]$ the tree obtained by {\em pruning} the branch $Q_a$ at $u$. Note that $\mu(Q) = \mu(\tilde{Q})$. Figure \ref{fig:pruning} depicts a network $Q$ with branch node $u$ and branch arc $a$, along with the tree $Q[a]$ obtained by pruning $Q_a$.

\begin{figure}[ht]
	\begin{center}
		\includegraphics[scale=0.5]{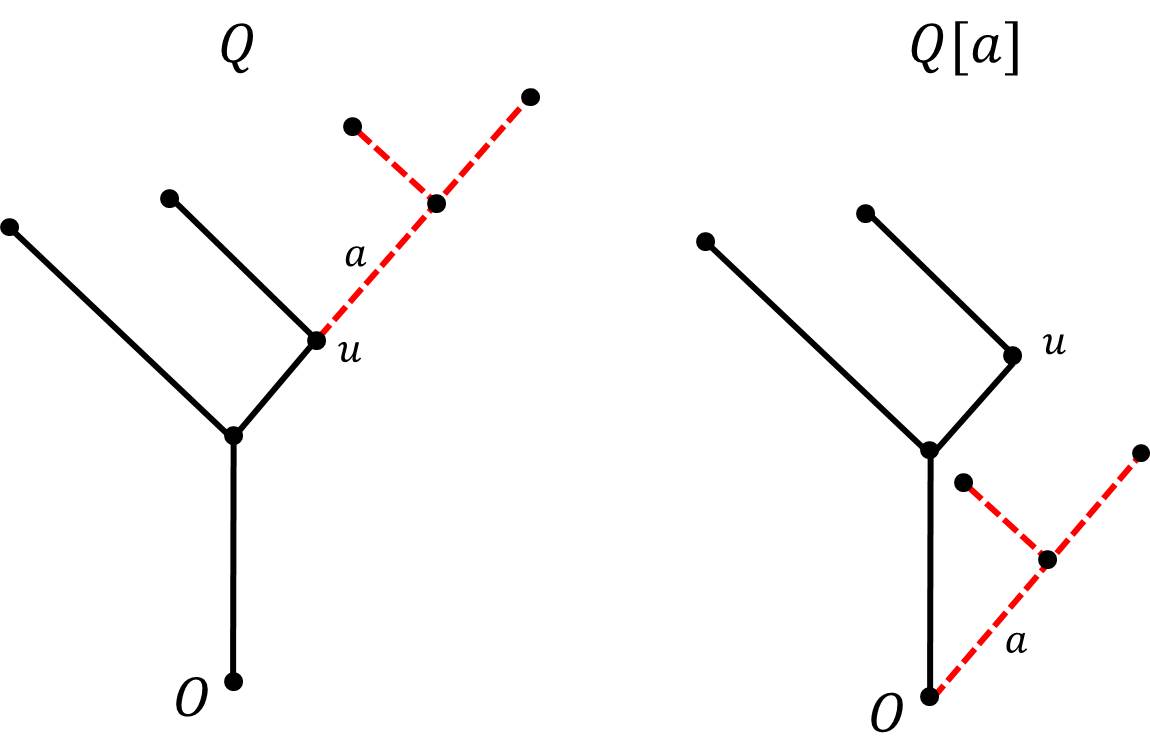}
		\caption{A tree $Q$ and the tree $Q[a]$ obtained by pruning $Q_a$.}
		\label{fig:pruning}
	\end{center}
\end{figure}

\begin{lemma}
	\label{lemma:pruning}
	Let $Q$ be a rooted network with branch node $u \neq O$ and branch arc $a$ incident to $u$, and suppose that there are no other branch nodes on the path between $O$ and $a$. Then $D(Q[a]) \le D(Q)$.
\end{lemma}

\begin{proof}
Let $Y$ be the complement in $Q_u$ of the branch $Q_a$ at $u$ and let $D_a$ and $D_Y$ be the average distance from $u$ to the leaf nodes of $Q_a$ and $Y$, respectively, weighted according to the EBD distribution on $Q_a$ and $Y$. The total contribution made to $D(Q)$ by the leaf nodes of $Q_a$ and $Y$ is 
\begin{equation}
e(Q_u) \left( \left( \frac{\mu(Q_a)}{\mu(Q_a) + \mu(Y)} \right) (d(u) + D_a) + \left( \frac{\mu(Y)}{\mu(Q_a) + \mu(Y)} \right) (d(u) + D_Y) \right).
\label{eq:incline}
\end{equation}
Similarly, the total contribution made by the associated leaf nodes in $Q[a]$ is
\begin{equation}
e(Q_u) \left( \left( \frac{\mu(Q_a)}{\mu(Q_a) + \mu(Y)+d(u)} \right) D_a + \left( \frac{\mu(Y)+d(u)}{\mu(Q_a) + \mu(Y)+d(u)} \right) (d(u) + D_Y) \right).
\label{eq:pruned-incline}
\end{equation}
The difference $\Delta$ between expressions (\ref{eq:incline}) and (\ref{eq:pruned-incline}) is given by
\[
\Delta = \frac{e(Q_u) \mu(Q_a) d(u)}{(\mu(Q_a) + \mu(Y))(\mu(Q_a) + \mu(Y) + d(u))}(D_a + \mu(Q_a) + \mu(Y) - D_Y).
\]
This expression for $\Delta$ is clearly non-negative because $\mu(Y) \ge D_Y$. Since the contribution to $D$ of all other leaf nodes in $Q$ not contained in $Y$ or $Q_a$ remains unchanged after pruning, it follows that $D(Q[a]) \le D(Q)$. 
\end{proof}

We now define the tree $\tilde{Q}$, and show that $D(\tilde{Q}) \le D(Q)$.

\begin{lemma} Let $Q$ be a tree network with root $O$, let $u$ be a node in $Q$ at maximum distance $d_{\max}$ from $O$ and let $\mathcal{P}$ be the path from $O$ to $u$. Let $\tilde{Q}$ be the tree obtained from $Q$ by successively pruning branches $Q_a$ at $v$ where $a \notin \mathcal{P}$ is a branch arc incident to some branch node $v$ in $\mathcal{P}$ and $v$ is chosen each time to have minimum distance from $O$. Then
	\begin{equation}
	D(Q) \ge D(\tilde{Q}) \ge d_{\max}^2/\mu(Q).
	\label{eq:D-est}
	\end{equation}
\end{lemma}

\begin{proof}
It is clear from Lemma \ref{lemma:pruning} that $D(Q) \ge D(\tilde{Q})$. To prove the second inequality, we simply note that the probability attached to $u$ in $\tilde{Q}$ under the EBD distribution is equal to $d_{\max}/\mu(Q)$, so the contribution to $D(\tilde{Q})$ made by $u$ is $(d_{\max}/\mu(Q))d_{\max} = d_{\max}^2/\mu(Q)$. 
\end{proof}

We may obtain an alternative lower bound on $V$ to the one given by Theorem \ref{thm:f&f} by considering the pure Hider strategy of chosing a node $u$ at maximum distance from $O$. This ensures a payoff of at least $d_{\max}(1+1/\rho)$, so that
\begin{equation}
V \ge d_{\max}(1+1/\rho).
\label{eq:f&f-lb}
\end{equation}

Note that (\ref{eq:f&f-lb}) holds for arbitrary networks, not only trees.

We can now state our main result for the approximation ratio of the RCPT in the find-and-fetch model.

\begin{thm}
	\label{thm:main-f&f}
	The approximation ratio of a RCPT in the find-and-fetch game on a tree is $\alpha = \alpha(\rho)$, where
	\[
	\alpha = 
	\begin{cases}
	2/(1+\rho) & \text{ if } \rho \le 1/3\\
	(1+\sqrt{1+1/\rho})/2    & \text{ if } \rho \ge 1/3.
	\end{cases}
	\]
\end{thm}

\begin{proof}
By equation (\ref{eq:R-RCPT}) and Theorem \ref{thm:f&f}, we have
\begin{align}
\frac{R_{RCPT}}{V} &\le \frac{\mu+ d_{\max}/\rho}{\mu+ D/\rho} \notag \\
&\le \frac{\rho + d_{\max}/\mu}{\rho+ d_{\max}^2/\mu^2} \mbox{ (by (\ref{eq:D-est}))} \notag \\
&= \frac{\rho+ z}{\rho+ z^2},
\label{eq:bound1}
\end{align}
where $z = d_{\max}/\mu$.

Using elementary calculus, we find that the expression (\ref{eq:bound1}) is maximized for $z = z_0:= \sqrt{\rho^2+\rho} - \rho$, where it takes the value $(1+\sqrt{1+1/\rho})/2$.

Combining (\ref{eq:R-RCPT}) with the alternative lower bound (\ref{eq:f&f-lb}) yields
\begin{align}
\frac{R_{RCPT}}{V} &\le \frac{\mu+ d_{\max}/\rho}{d_{\max}(1+1/\rho)} \notag \\
&= \frac{\rho + z}{(\rho+ 1)z}.
\label{eq:bound2}
\end{align}

Expression (\ref{eq:bound2}) is a decreasing function of $z$. Writing $f_1(z)$ and $f_2(z)$ for expressions (\ref{eq:bound1}) and (\ref{eq:bound2}), respectively, we have $R_{RCPT}/V \le r := \max_{z \ge 0} \min \{f_1(z),f_2(z)\}$. Since $f_2$ is decreasing in $z$, the from $r$ takes splits into two cases which depend on the value of $\rho$: these cases are $f_2(z_0) \ge f_1(z_0)$ or $f_2(z_0) \le f_1(z_0)$, as depicted in Figure \ref{fig:tree-graphs}. 

\begin{figure}[ht]
	\begin{center}
		\includegraphics[scale=0.7]{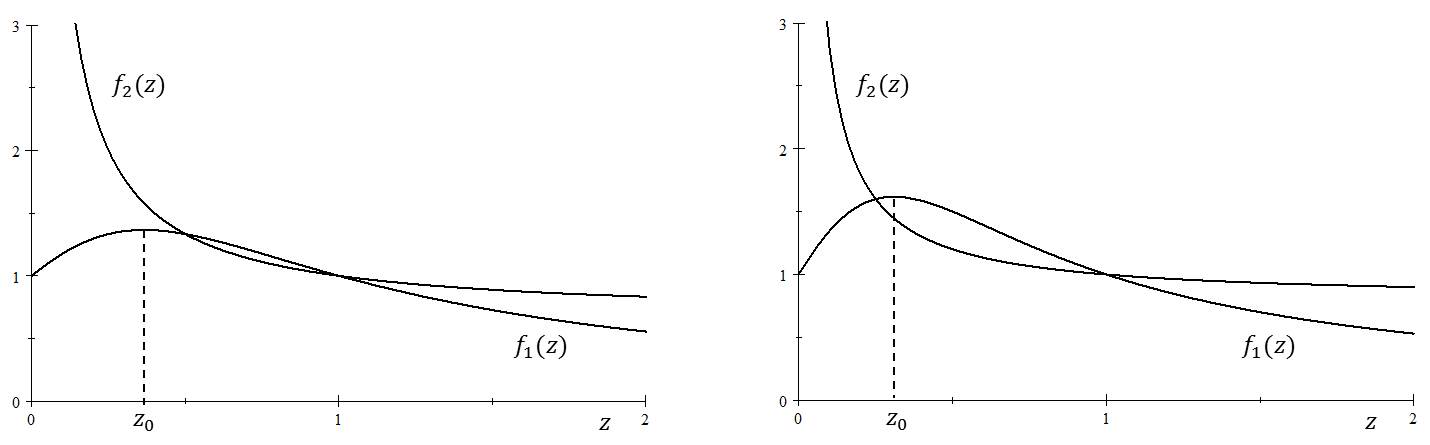}
		\caption{The graphs of $f_1(z)$ and $f_2(z)$ for $\rho=1/2$ (left) and $\rho=1/4$ (right).}
		\label{fig:tree-graphs}
	\end{center}
\end{figure}

It is easy to check that the first case occurs when $\rho \ge 1/3$, in which case $r = f(z_0)=(1+\sqrt{1+1/\rho})/2$. In the second case (when $\rho \le 1/3$), the bound $r$ is given by the value of two functions $f_1$ and $f_2$ when they intersect for the first time. The two intersections can be shown to be at $z=\rho$ and $z=1$, so at the first intersection, $z=\rho$, we have $r = f_1(\rho) = 2/(1+\rho)$. The theorem follows. 

\end{proof}

Figure \ref{fig:tree-alpha} is a plot of the approximation ratio $\alpha(\rho)$ of a RCPT for tree networks.
\begin{figure}[ht]
	\begin{center}
		\includegraphics[scale=0.7]{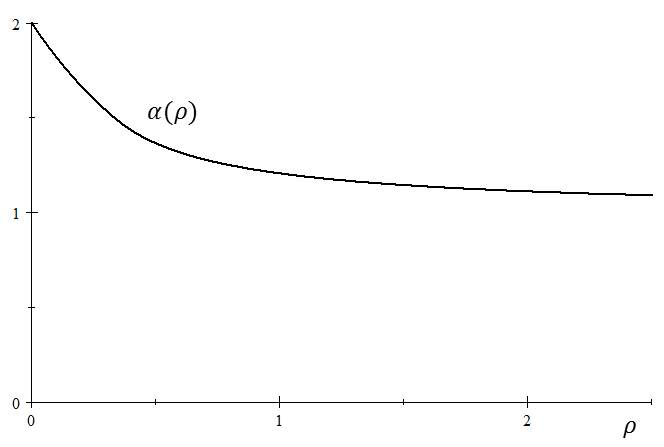}
		\caption{The graph of $\alpha(\rho)$ for tree networks.}
		\label{fig:tree-alpha}
	\end{center}
\end{figure}

Note that as $\rho \rightarrow \infty$, the approximation ratio $\alpha$ approaches $(1+\sqrt{1+0})/2=1$, and so we recover the result of \cite{Gal79} that the RCPT is optimal for trees. Also note that $\alpha$ is bounded above by $2$, with $\alpha(\rho) \rightarrow 2$ as $\rho \rightarrow 0$. When $\rho = 1$, so that the Searcher's return speed is the same as her searching speed, $\alpha = (1+\sqrt{2})/2 \approx 1.21$.

\subsection{Eulerian Networks}
We conclude our analysis of the performance of the RCPT by examining the find-and-fetch game on Eulerian networks. Nothing is known about the solution of this game on such networks, as the work of \cite{Alperna} restricts its attention to trees. 

For Eulerian networks, since the expected search time of any point under a RCPT is $\mu/2$, the maximum expected payoff $R_{RCPT}$ of a RCPT is easily seen to be
\begin{equation}
R_{RCPT} = \mu/2 + d_{\max}/\rho.
\label{eq:R-RCPT2}
\end{equation}

In order to estimate the efficiency of the RCPT we may use the lower bound (\ref{eq:f&f-lb}) given in the proof of Theorem \ref{thm:main-f&f}, but this bound is not very good if $d_{\max}$ is small: in this case the Hider benefits from randomization. So we construct another bound, based on the {\em uniform} strategy $\nu$ of the Hider: that is the mixed strategy that chooses a subset $A$ of $Q$ with probability proportional to the length $\mu(A)$ of $A$.

\begin{lemma}
	Consider the find-and-fetch game on a rooted Eulerian network, $Q$. The uniform strategy $\nu$ of the Hider guarantees the following lower bound on the value $V$.
	\begin{equation}
	V \ge \mu/2 + d_{\max}^2/(\rho \mu).
	\label{eq:f&f-lb2}
	\end{equation}
\end{lemma}

\begin{proof}
First note that against  $\nu$ any CPT is a best response (that is, it is a strategy with minimal expected payoff against $\nu$). The expected search time of a CPT against $\nu$ is $\mu/2$ (this is fairly obvious, but can be found, for example, in \cite{Gal79}) so the payoff of the find-and-fetch game is given by $R(CPT,\nu) = \mu/2 + \bar{d}$, where $\bar{d}$ is the expected distance from $O$ to all points in $Q$, with respect to $\nu$. That is,
\[
\bar{d} = \int_{x \in Q} \frac{d(x)}{\rho} \cdot \frac{dx}{\mu} .
\]
Let $\mathcal{P}_1$ be a shortest path from $O$ to a point $y$ at maximum distance from $O$. Since $Q$ is Eulerian, it is 2-arc-connected, and it follows that there is another path $\mathcal{P}_2$ from $O$ to $y$ which is edge disjoint from $\mathcal{P}_1$ and has length at least $d_{\max}$ (otherwise $y$ would not be at maximum distance from $O$). It follows that
\begin{align*}
\bar{d} &\ge \int_{x \in \mathcal{P}_1} \frac{d(x)}{\rho} \cdot \frac{dx}{\mu}  + \int_{x \in \mathcal{P}_2} \frac{d(x)}{\rho} \cdot \frac{dx}{\mu}  \\
&\ge 2 \int_{x \in \mathcal{P}_1} \frac{d(x)}{\rho} \cdot \frac{dx}{\mu} \\
&= d_{\max}^2/(\rho \mu).
\end{align*}
Inequality (\ref{eq:f&f-lb2}) follows.
\end{proof}

Combining our expression (\ref{eq:R-RCPT2}) for the payoff of a RCPT with our two lower bounds, (\ref{eq:f&f-lb}) and (\ref{eq:f&f-lb2}) on the value result in the following.

\begin{thm}
	The approximation ratio of a RCPT in the find-and-fetch game on a Eulerian network is $\alpha = \alpha(\rho)$, where
	\[
	\alpha = \frac{3+\rho+\sqrt{\rho^2+1}}{2+2\rho}.
	\]
\end{thm}

\begin{proof}
By (\ref{eq:R-RCPT2}) and (\ref{eq:f&f-lb2}), we have
\begin{align}
\frac{R_{RCPT}}{V} & \le \frac{\mu/2+ d_{\max}/\rho}{\mu/2 + d_{\max}^2/{(\rho \mu)}} \notag \\
& = \frac{\rho/2 + z}{\rho/2 + z^2},
\label{eq:bound3}
\end{align}
where $z = d_{max}/\mu$, as before. Expression (\ref{eq:bound3}) has a unique maximum at  $z = z_1 = (\sqrt{2\rho+ \rho^2}-\rho)/2$, at which point expression (\ref{eq:bound3}) is equal to $(1+\sqrt{1+2/\rho})/2$.

Similarly, using (\ref{eq:f&f-lb}), we obtain
\begin{align}
\frac{R_{RCPT}}{V} & \le \frac{\mu/2+ d_{\max}/\rho}{\mu/2 + d_{\max}^2/{\rho \mu}} \notag\\
& = \frac{\rho/2 + z}{z(1+\rho)}.
\label{eq:bound4}
\end{align}
Write $g_1(z)$ and $g_2(z)$ for expressions (\ref{eq:bound3}) and (\ref{eq:bound4}), respectively, so that $R_{RCPT}/V \le \max_{z \ge 0} \min\{g_1(z),g_2(z)\}$. Substituting $z=z_1$ into $g_2(z)$ we find that $g_2(z_1) = (1+1/(1+\rho)+\sqrt{2\rho+\rho^2}/(1+\rho))/2$. We calculate the difference between $g_1(z)$ and $g_2(z)$ at $z=z_1$.
\[
g_1(z_1) - g_2(z_1) = \frac{ \sqrt{(1+\rho)^2(2+\rho)/\rho} - (1+\sqrt{2\rho+ \rho^2}) }{2+2\rho}.
\]
The expression on the right-hand side is non-negative if and only if the difference between the squares of the two terms in the numerator are non-negative.  That is, if
\begin{align*}
(1+\rho)^2(2+\rho)/\rho - (1+\sqrt{2\rho+ \rho^2})^2 &\ge 0 \mbox{ which holds if and only if} \\ 
\sqrt{(1+\rho)^4} - \sqrt{\rho^4 + 2\rho^3} &\ge 0,
\end{align*}
which is clearly always true, so $g_1(z_1) \ge g_2(z_1)$ for any $\rho$. Since also $g_2$ is decreasing with $g_2(z) \rightarrow \infty$ as $z \rightarrow 0$ and $g_1(0) = 1$, it follows that $g_1$ and $g_2$ intersect at some $z \le z_1$, and $\min\{g_1(z),g_2(z)\}$ is maximized at this value of $z$, as shown in Figure \ref{fig:eul-graph}.
\begin{figure}[ht]
	\begin{center}
		\includegraphics[scale=0.7]{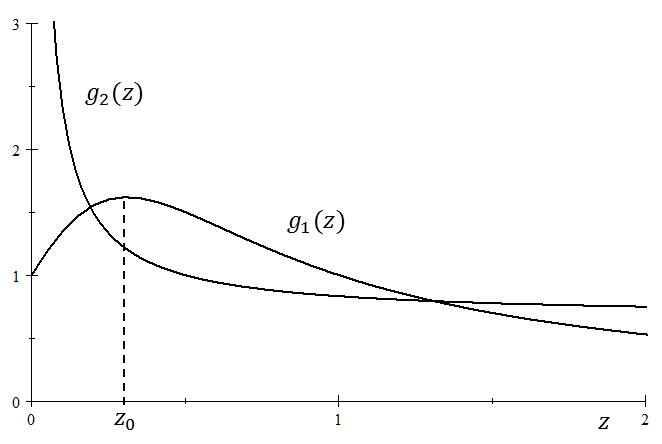}
		\caption{The graphs of $g_1(z)$ and $g_2(z)$ for $\rho=1/2$.}
		\label{fig:eul-graph}
	\end{center}
\end{figure}
It is easy to show that $g_1$ and $g_2$ intersect at two points, and the smaller of those two intersections, at $z=(1+\rho-\sqrt{1+\rho^2})/2$, maximizes $\min\{g_1(z),g_2(z)\}$. Substituting this into $g_1$ or $g_2$ gives the approximation ratio in the statement of the theorem. 
\end{proof}

As in the case of trees, we may recover the result of \cite{Gal79} for Eulerian networks, by writing the limit of the approximation ratio of a RCPT as
\[
\alpha = \frac{3+\rho+\sqrt{\rho^2+1}}{2+2\rho} = \frac{1}{2} + \frac{1}{1+\rho} + \sqrt{1/4-\frac{\rho/2}{(1+\rho)^2}} \rightarrow 1,
\]
as $\rho \rightarrow \infty$.

Figure \ref{fig:eul-alpha} is a plot of the approximation ratio $\alpha(\rho)$ of a RCPT for Eulerian networks.
\begin{figure}[ht]
	\begin{center}
		\includegraphics[scale=0.7]{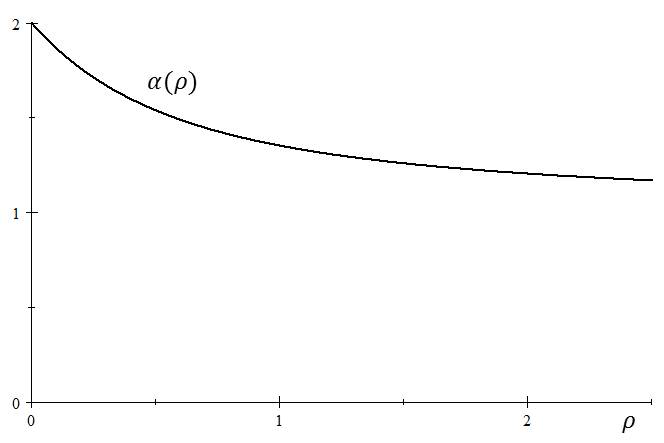}
		\caption{The graph of $\alpha(\rho)$ for Eulerian networks.}
		\label{fig:eul-alpha}
	\end{center}
\end{figure}

Note that $\alpha$ is decreasing in $\rho$, so the approximation ratio is bounded above by $\alpha(0) = 2$. When $\rho=1$, so the return speed is the same as the searching speed, $\alpha = 1+\sqrt{2}/4 \approx 1.35$.

\section{Conclusion}

We have examined the performance of a RCPT for Gal's classic search game on a network and for the find-and-fetch game. In the former game we gave a simple formula for the approximation ratio of a RCPT for any network and in the latter game we did the same for trees and Eulerian networks. Since the time to compute a RCPT is polynomial in the number of nodes of the network, this work provides efficiently computable, easily implemented search strategies for games that in general have complicated exact solutions.

The calculation of the approximation ratio of a RCPT in the find-and-fetch game required the use of a new technique called ``pruning''. In order to extend this work on the find-and-fetch game to arbitrary networks, it is likely that further new techniques will have to be developed. Future work could also be directed towards the version of Gal's game where the Searcher has an arbitrary starting point, as in \cite{DG} and \cite{Alpern-arb}. In this game, one could examine the approximation ratio of a {\em Random Chinese Postman Path}, that is Searcher strategy that chooses with equal probability a path in the network of minimum length and its reverse.

%%%%%%%%%%%%%%%%%
\end{document}